\documentclass{amsart}
\usepackage{longtable}
\usepackage{amssymb,amsthm}
\usepackage{hyperref}
\newtheorem{theorem}{Theorem}[section]
\newtheorem{lemma}[theorem]{Lemma}

\begin{document}
\title[Base size on subsets]{A formula for the base size of the symmetric group in its action on subsets}

\author[G. Mecenero]{Giovanni Mecenero}
\address{Giovanni Mecenero, Dipartimento di Matematica \lq\lq Tullio Levi-Civita\rq\rq,\newline
 University of Padova, Via Trieste 53, 35121 Padova, Italy} 
\email{giovanni.mecenero@studenti.unipd.it}

\author[P. Spiga]{Pablo Spiga}
\address{Pablo Spiga, Dipartimento di Matematica Pura e Applicata,\newline
 University of Milano-Bicocca, Via Cozzi 55, 20126 Milano, Italy} 
\email{pablo.spiga@unimib.it}
\maketitle
\begin{abstract}
Given two positive integers $n$ and $k$, we obtain a formula for the base size of the symmetric group of degree $n$ in its action on $k$-subsets. Then, we use this formula to compute explicitly the base size for each $n$ and  for each $k\le 14$.
\end{abstract}
\section{Introduction}\label{introduction}

Let $G$ be a permutation group on $\Omega$. For $\Lambda = \{\omega_1, \ldots, \omega_k \} \subseteq \Omega$, we write
 $G_{(\Lambda)}$  for the \textit{\textbf{pointwise
stabilizer}} of $\Lambda$ in $G$. If $G_{(\Lambda)} = \{1\}$, then we say that $\Lambda$ is a \textit{\textbf{base}}. The size of a smallest possible base is known as the \textit{\textbf{base size}} of $G$ and it is customary to denote it by $\mathrm{b}(G)$ or (more precisely) by $\mathrm{b}_{\Omega}(G)$. Bases of small cardinality are relevant in computational group theory, because they are used in many
algorithms for dealing with permutation groups.

Much of the research on base sizes is focused on primitive groups. There are various reasons for that. First, the classic result of Jordan~\cite{jordan} bounds from above the cardinality of a primitive group in terms of its base size.  Second, in the 90s  Cameron and Kantor~\cite{cameron} have conjectured that
there exists an absolute constant $b$ with $\textrm{b}(G) \le b$, for every almost simple primitive
group in a non-standard action; we refer to~\cite{cameron} for undefined terminology.  This conjecture was settled in the positive in~\cite{B5,B6,B7,B8} with $b=7$. In turn, this result has stimulated two lines of research: determining the base size of almost simple primitive groups in their standard actions; determining, for each $b\le 7$, the almost simple primitive groups in non-standard actions having base size exactly $b$.

The symmetric group $\mathrm{Sym}(n)$ has two types of standard actions: the action on uniform partitions and the action on subsets of $\{1,\ldots,n\}$. The base size of $\mathrm{Sym}(n)$ for the action  on uniform partitions has been determined in~\cite{mosp}.
In this paper, we are interested in the base size of the symmetric group $\mathrm{Sym}(n)$ in its natural action on the collection of $k$-subsets of $\{1, \ldots, n\}$. Hence this paper is a contribution to the determination of the base size of almost simple primitive groups in their standard actions.

For simplicity, we denote this number with $b(n,k)$. Since the action of $\mathrm{Sym}(n)$ on $k$-subsets is permutation equivalent to the action on $(n-k)$-subsets, we have $$b(n,k)=b(n,n-k)$$ and hence we assume $2k\le n$.

There is a natural graph theoretic interpretation of $b(n,k)$, which makes  $b(n,k)$ of interest to group and to graph theorists. A \textit{\textbf{determining set}} of a graph $\Gamma$ is a subset $S$ of its vertices for which the only automorphism of $\Gamma$ that fixes every vertex in $S$ is the identity. The minimum cardinality of a determining set for $\Gamma$ is said to be the \textit{\textbf{determining number}} of $\Gamma$. Now, the \textit{\textbf{Kneser}} graph $K(n,k)$ is the graph having vertex set the collection of all $k$-subsets of $\{1,\ldots,n\}$ where two distinct $k$-subsets are declared to be adjacent if they are disjoint. Since the automorphism group of $K(n,k)$ is the symmetric group $\mathrm{Sym}(n)$ in its action on $k$-subsets (see~\cite[Theorem~7.8.2]{Godsil}), we deduce that the determining number of $K(n,k)$ equals $b(n,k)$. Therefore, our results can be interpreted in terms of the determining number of Knerser graphs.

There are many partial results on $b(n,k)$, see~\cite{boutin,CGGM,DD,halasi}. For instance, Halasi~\cite[Theorem~3.2]{halasi} has proved that
\begin{equation}\label{equation}
b(n,k)=\left\lceil\frac{2(n-1)}{k+1}\right\rceil,
\end{equation}
when $n\ge \lfloor k(k+1)/2\rfloor+1$. Strictly speaking, this formula for $b(n,k)$ is proved in~\cite{halasi} when $n\ge k^2$ and has been improved in~\cite{CGGM} to $n\ge \lfloor k(k+1)/2\rfloor+1 $. This result has been improved further in~\cite{boutin,CGGM,halasi}, but currently there is no explicit formula for $b(n,k)$, valid for every value of $n$ and $k$. Our current methods do not seem able to determine $b(n,k)$ when $k$ is asymptotically larger than $\sqrt{n}$.

Using the principle of inclusion-exclusion, we prove an implicit formula for $b(n,k)$ in terms of integer partitions of $n$. It is not clear at the moment if this formula can be used to extract substantial new information on $b(n,k)$. However, besides  theoretic interest, we have implemented this formula in a computer and we are reporting in Table~\ref{table1} the values of $b(n,k)$, for every $k\le 14$.  We hope that these values can be of some help to shed some light on $b(n,k)$, when $k$ is large. 

Observe that in Table~\ref{table1}, for a given $k\le 14$, we are reporting only the values of $b(n,k)$ when $n\le \lfloor k(k+1)/2\rfloor$, because when $n\ge \lfloor k(k+1)/2\rfloor+1$ we may simply use~\eqref{equation} to compute $b(n,k)$. In particular, when $k=1$, we have $\lfloor k(k+1)/2\rfloor=1\le 2k$ and hence  $b(n,k)=\lceil 2(n-1)/2\rceil=n-1$, for every $n\ge 2$. When $k=2$, we have $\lfloor k(k+1)/2\rfloor=3\le 2k$ and hence  $b(n,k)=\lceil2(n-1)/3\rceil$, for every $n\ge 4$.  For this reason, in Table~\ref{table1}, we are only including the values of $k\ge 3$.

\begin{theorem}\label{thrm:main}
Let $n$ and $k$ be positive integers with $2k\le n$. Then $b(n,k)$ is the smallest positive integer $\ell$ such that
\begin{align}\label{eq:5}
\sum_{\substack{\pi\hbox{ partition of }n\\
\pi=(1^{c_1},2^{c_2},\ldots,n^{c_n})}}(-1)^{n-\sum_{i=1}^nc_i}\frac{n!}{\prod_{i=1}^ni^{c_i}c_i!}\left(\sum_{\substack{\eta\hbox{ partition of }k\\
\eta=(1^{b_1},2^{b_2},\ldots,k^{b_k})}}\prod_{j=1}^k{c_j\choose b_j}\right)^\ell&\ne 0.
\end{align}
\end{theorem}
A result similar  to Theorem~\ref{thrm:main} was very recently determined independently by Coen del Valle and  Colva Roney-Dougal~\cite{arxiv}, their proof is remarkably different from ours.
\begin{longtable}{c|c|c|c|c|c|c|c|c|c|c|c|c}\hline
$n\backslash k$&3&4&5&6&7&8&9&10&11&12&13&14\\\hline
6 &3&&&&&&&&&&\\
7 &-&&&&&&&&&&\\
8 &-&3&&&&&&&&&\\
9 &-&4&&&&&&&&&\\
10&-&4&4&&&&&&&&\\
11&-&-&4&&&&&&&&\\
12&-&-&4&4&&&&&&&\\
13&-&-&5&4&&&&&&&\\
14&-&-&5&5&4&&&&&&\\
15&-&-&5&5&4&&&&&&\\
16&-&-&-&5&5&4&&&&&\\
17&-&-&-&5&5&5&&&&&&\\
18&-&-&-&6&5&5&5&&&&&\\
19&-&-&-&6&5&5&5&&&&&\\
20&-&-&-&6&6&5&5&5&&&&\\
21&-&-&-&6&6&5&5&5&&&&\\
22&-&-&-&-&6&6&5&5&5&&&\\
23&-&-&-&-&6&6&6&5&5&&&\\
24&-&-&-&-&6&6&6&5&5&5&&\\
25&-&-&-&-&7&6&6&6&5&5&&\\
26&-&-&-&-&7&6&6&6&5&5&5&\\
27&-&-&-&-&7&7&6&6&6&5&5&\\
28&-&-&-&-&7&7&6&6&6&6&5&5\\
29&-&-&-&-&-&7&7&6&6&6&6&5\\
30&-&-&-&-&-&7&7&6&6&6&6&6\\
31&-&-&-&-&-&7&7&7&6&6&6&6\\
32&-&-&-&-&-&8&7&7&6&6&6&6\\
33&-&-&-&-&-&8&7&7&7&6&6&6\\
34&-&-&-&-&-&8&8&7&7&6&6&6\\
35&-&-&-&-&-&8&8&7&7&7&6&6\\
36&-&-&-&-&-&8&8&7&7&7&6&6\\
37&-&-&-&-&-&-&8&8&7&7&7&6\\
38&-&-&-&-&-&-&8&8&7&7&7&6\\
39&-&-&-&-&-&-&8&8&8&7&7&7\\
40&-&-&-&-&-&-&9&8&8&7&7&7\\
41&-&-&-&-&-&-&9&8&8&8&7&7\\
42&-&-&-&-&-&-&9&8&8&8&7&7\\
43&-&-&-&-&-&-&9&9&8&8&7&7\\
44&-&-&-&-&-&-&9&9&8&8&8&7\\
45&-&-&-&-&-&-&9&9&8&8&8&7\\
46&-&-&-&-&-&-&-&9&9&8&8&8\\
47&-&-&-&-&-&-&-&9&9&8&8&8\\
48&-&-&-&-&-&-&-&9&9&9&8&8\\
49&-&-&-&-&-&-&-&9&9&9&8&8\\
50&-&-&-&-&-&-&-&10&9&9&8&8\\
51&-&-&-&-&-&-&-&10&9&9&9&8\\
52&-&-&-&-&-&-&-&10&9&9&9&8\\
53&-&-&-&-&-&-&-&10&10&9&9&8\\
54&-&-&-&-&-&-&-&10&10&9&9&9\\
55&-&-&-&-&-&-&-&10&10&9&9&9\\
56&-&-&-&-&-&-&-& -&10&10&9&9\\
57&-&-&-&-&-&-&-& -&10&10&9&9\\
58&-&-&-&-&-&-&-& -&10&10&9&9\\
59&-&-&-&-&-&-&-& -&10&10&10&9\\
60&-&-&-&-&-&-&-& -&11&10&10&9\\
61&-&-&-&-&-&-&-& -&11&10&10&9\\
62&-&-&-&-&-&-&-& -&11&10&10&10\\
63&-&-&-&-&-&-&-& -&11&11&10&10\\
64&-&-&-&-&-&-&-& -&11&11&10&10\\
65&-&-&-&-&-&-&-& -&11&11&10&10\\
66&-&-&-&-&-&-&-& -&11&11&10&10\\
67&-&-&-&-&-&-&-& -& -&11&11&10\\
68&-&-&-&-&-&-&-& -& -&11&11&10\\
69&-&-&-&-&-&-&-& -& -&11&11&10\\
70&-&-&-&-&-&-&-& -& -&11&11&11\\
71&-&-&-&-&-&-&-& -& -&12&11&11\\
72&-&-&-&-&-&-&-& -& -&12&11&11\\
73&-&-&-&-&-&-&-& -& -&12&11&11\\
74&-&-&-&-&-&-&-& -& -&12&11&11\\
75&-&-&-&-&-&-&-& -& -&12&12&11\\
76&-&-&-&-&-&-&-& -& -&12&12&11\\
77&-&-&-&-&-&-&-& -& -&12&12&11\\
78&-&-&-&-&-&-&-& -& -&12&12&11\\
79&-&-&-&-&-&-&-& -& -&-&12&12\\
80&-&-&-&-&-&-&-& -& -&-&12&12\\
81&-&-&-&-&-&-&-& -& -&-&12&12\\
82&-&-&-&-&-&-&-& -& -&-&12&12\\
83&-&-&-&-&-&-&-& -& -&-&12&12\\
84&-&-&-&-&-&-&-& -& -&-&13&12\\
85&-&-&-&-&-&-&-& -& -&-&13&12\\
86&-&-&-&-&-&-&-& -& -&-&13&12\\
87&-&-&-&-&-&-&-& -& -&-&13&12\\
88&-&-&-&-&-&-&-& -& -&-&13&13\\
89&-&-&-&-&-&-&-& -& -&-&13&13\\
90&-&-&-&-&-&-&-& -& -&-&13&13\\
91&-&-&-&-&-&-&-& -& -&-&13&13\\
92&-&-&-&-&-&-&-& -& -&-&-&13\\
93&-&-&-&-&-&-&-& -& -&-&-&13\\
94&-&-&-&-&-&-&-& -& -&-&-&13\\
95&-&-&-&-&-&-&-& -& -&-&-&13\\
96&-&-&-&-&-&-&-& -& -&-&-&13\\
97&-&-&-&-&-&-&-& -& -&-&-&14\\
98&-&-&-&-&-&-&-& -& -&-&-&14\\
99&-&-&-&-&-&-&-& -& -&-&-&14\\
100&-&-&-&-&-&-&-& -& -&-&-&14\\
101&-&-&-&-&-&-&-& -& -&-&-&14\\
102&-&-&-&-&-&-&-& -& -&-&-&14\\
103&-&-&-&-&-&-&-& -& -&-&-&14\\
104&-&-&-&-&-&-&-& -& -&-&-&14\\
105&-&-&-&-&-&-&-& -& -&-&-&14\\
\caption{Some values for $b(n,k)$}\label{table1}
\end{longtable}

\section{Proof of Theorem~\ref{thrm:main}}

Let $n, k$ and $\ell$ be positive integers with $1\le k\le n/2$. We let 
$$F={\{1,\ldots,n\}\choose k}$$
be the collection of all $k$-subsets of $\{1,\ldots,n\}$ and we let $F^\ell$ be the collection of all $\ell$-tuples of $k$-subsets of $\{1,\ldots,n\}$. In particular, $$|F^\ell|={n\choose k}^\ell.$$

For each $g\in \mathrm{Sym}(n)$, we let 
$$F_g=\{\alpha\in F\mid \alpha^g=\alpha\}$$
be the collection of all $k$-subsets of $\{1,\ldots,n\}$  fixed setwise by $g$. 
Therefore, the cartesian product $F_g^\ell$ is the collection of all $\ell$-tuples of $k$-subsets of $\{1,\ldots,n\}$  fixed setwise by $g$. 
For instance, for each $1\le i<j\le n$, $F_{(i\,j)}$ is the collection of all $k$-subsets of $\{1,\ldots,n\}$ fixed by the transposition swapping $i$ and $j$.
 Moreover we let $$H_\ell=\{(\alpha_1,\ldots,\alpha_\ell)\in F^\ell\mid \textrm{if }\alpha_i^g=\alpha_i \,\forall i, \textrm{ then }g=1\}$$ denote the collection of all $\ell$-tuples of $k$-subsets of $\{1,\ldots,n\}$ which are only fixed by the identity, and $h_\ell = |H_\ell|$.\\
Since each element of $F^\ell$ is either fixed by some non-identity element of $\mathrm{Sym}(n)$ or is only fixed by the identity, exclusively, we have
\begin{equation}\label{verajail}F^\ell \setminus H_\ell =\bigcup_{g\in\mathrm{Sym}(n)\setminus\{1\}}F_g^\ell.
\end{equation}
Observe that
\begin{equation}\label{key}
    h_\ell = 0 \hbox{ if and only if }\ell<b(n,k).
\end{equation}
In fact, by definition of $b(n,k)$, there exists an $\ell$-tuple of $k$-subsets which is only fixed by the identity if and only if $\ell\geq b(n,k)$.

\begin{lemma}\label{lem:2.1}
    Given a non-identity permutation $g\in \mathrm{Sym}(n)$, there exists a transposition $\tau$ such that, for every $S\subseteq\{1,\ldots,n\}$ fixed by $g$, $S$ is also fixed by $\tau$.
\end{lemma}

\begin{proof}
    Let $(a_1, a_2,\ldots,a_i)$ be one of the cycles of $g$ in its decomposition in disjoint cycles. As $g$ is not the identity, we may suppose that $i\ge 2$. Let $\tau$ be the transposition $(a_1, a_2)$.
    
Let $S$ be a subset of $\{1,\ldots,n\}$ with $S^g=S$.    Then either $S\cap \{a_1,\ldots,a_i\}=\emptyset$ or $S\cap \{a_1,\ldots,a_i\}=\{a_1,\ldots,a_i\}$. In either case,  $a_1$ and $a_2$ are either both in $S$ or both not in $S$. Therefore $\tau$ fixes $S$.
\end{proof}

From~\eqref{verajail} and Lemma~\ref{lem:2.1}, we obtain
\begin{align}\label{eq:1}
F^\ell\setminus H_\ell&=\bigcup_{1\le i<j\le n}F_{(i\,j)}^\ell.
\end{align}

Now, we let ${\{1,\ldots,n\}\choose 2}$ denote the set of all $2$-subsets of $\{1,\ldots,n\}$. From~\eqref{eq:1} and the definition of $h_\ell$, using inclusion-exclusion, we obtain
\begin{align}\label{eq:2}
{n\choose k}^\ell-h_\ell=|F^\ell\setminus H_\ell|=\sum_{\emptyset\ne\Gamma\subseteq{\{1,\ldots,n\}\choose 2}}(-1)^{|\Gamma|-1}\left|\bigcap_{\{i,j\}\in \Gamma}F_{(i\,j)}^\ell\right|.
\end{align}

Now, given a subset $\emptyset\ne \Gamma\subseteq{\{1,\ldots,n\}\choose 2}$, we write $$F_\Gamma^\ell=\bigcap_{\{i,j\}\in \Gamma}F_{(i\,j)}^\ell$$
and, with a slight abuse of terminology, we let $F_\emptyset^\ell:=F^\ell$. With this notation, from~\eqref{eq:2}, we get
\begin{align}\label{eq:33vera}
h_\ell&=\sum_{\Gamma\subseteq{\{1,\ldots,n\}\choose 2}}(-1)^{|\Gamma|}|F_\Gamma^\ell|.
\end{align}

In what follows, we identify  $\Gamma\subseteq {\{1,\ldots,n\}\choose 2}$ with a graph on $\{1,\ldots,n\}$ having edge set $\Gamma$. In particular, we borrow some notation from graph theory.

Given $\Gamma\subseteq {\{1,\ldots,n\}\choose 2}$, we let $\pi(\Gamma)$ be the partition of $n$ where the parts are the cardinalities of the connected components of $\Gamma$. In other words, let $X_{1},X_2,\ldots,X_t$ be the connected components of $\Gamma$ ordered with $|X_1|\ge |X_2|\ge \cdots \ge |X_t|$. Then $$\pi(\Gamma):=(|X_1|,|X_2|,\ldots,|X_t|).$$

We now make two important remarks.
\begin{lemma}\label{lemma:1}
Let $\Gamma$ be a graph having vertex set $\{1,\ldots,n\}$ and having connected components $X_1,\ldots,X_t.$ Then $$\langle (i\,j)\mid \{i,j\}\in \Gamma\rangle=\mathrm{Sym}(X_1)\times\mathrm{Sym}(X_2)\times\cdots\times\mathrm{Sym}(X_t).$$
\end{lemma}
In other words, the group generated by the transpositions corresponding to elements in $\Gamma$ generate a group which is a direct product of symmetric groups.
\begin{proof}
For each $i\in \{1,\ldots,t\}$, let $\Gamma_i$ be the restriction of $\Gamma$ to $X_i$.
Suppose first  $t=1$, that is, $\Gamma$ is connected. In this case, we prove the lemma by indution on $n$. 
When $n = 1$, the lemma holds true because the group generated by the empty set is the identity symmetric group $\mathrm{Sym}(1)$.
Suppose  $n\geq 2$. Let $\{a,b\}$ be a leaf of a spanning tree for $\Gamma$. Without loss of generality we may suppose $n\in \{a,b\}$ and that $n$ is a leaf of the spanning tree. The restriction $\tilde\Gamma$ of $\Gamma$ to $\{1,\ldots,n-1\}$ is connected and hence, by our inductive hypothesis, we have 
    $$\langle (i\,j)\mid \{i,j\}\in \tilde\Gamma\rangle = \mathrm{Sym}(n-1).$$
Now,
    $$\mathrm{Sym}(n)\ge\langle (i\,j)\mid \{i,j\}\in \Gamma\rangle \geq \langle\mathrm{Sym}(n-1),(a,b)\rangle=\mathrm{Sym}(n).$$
    
Assume now $t\ge 2$.    As $X_i\cap X_j=\emptyset$ for $i\neq j$ and as every edge in $\Gamma$ is in $\Gamma_i$ for some $i$, we deduce
    \begin{align*}
        \langle (i\,j)\mid \{i,j\}\in \Gamma\rangle=\langle (i\,j)\mid \{i,j\}\in \Gamma_1\rangle\times\cdots\times\langle (i\,j)\mid \{i,j\}\in \Gamma_t\rangle.
    \end{align*}
    Now, the lemma follows by applying the case of connected graphs.
\end{proof}

 From Lemma~\ref{lemma:1}, it immediately follows that, for every $\Gamma_1,\Gamma_2\subseteq{\{1,\ldots,n\}\choose 2}$ with $\pi(\Gamma_1)=\pi(\Gamma_2)$, we have $|F_{\Gamma_1}^\ell|=|F_{\Gamma_2}^\ell|$. Since the cardinality of these sets depends only on an integer partition, for each partition $\pi$ of $n$, we let $f_\pi^\ell$ be the cardinality of $|F_\Gamma^\ell|$, where $\Gamma$ is an arbitrary graph with $\pi=\pi(\Gamma)$.

Fix $\pi$ a partition of $n$. We write $\pi$ in ``exponential'' notation, that is, $\pi=(1^{c_1},2^{c_2},\ldots,n^{c_n})$ where $c_i$ denotes the number of parts in $\pi$ equal to $i$. 
\begin{lemma}\label{lemma:3}
Given an integer partition $\pi=(1^{c_1},\cdots,n^{c_n})$ of $n$, we have
\begin{align}\label{eq:4}\sum_{\substack{\Gamma\subseteq{\{1,\ldots,n\}\choose 2}\\ \pi(\Gamma)=\pi}}(-1)^{|\Gamma|}=(-1)^{n-\sum_{i=1}^nc_i}\frac{n}{\prod_{i=1}^ni^{c_i}c_i!}.
\end{align}
\end{lemma}
\begin{proof}
    First, we show \eqref{eq:4} in the special case $\pi = (1^0,2^0,\dots,n^1)$, that is, the trivial partition consisting of one part of size $n$. In particular we show
    $$\sum_{\substack{\Gamma\subseteq\binom{\{1,\dots,n\}}{2}\\ \pi(\Gamma)=(n) }}{(-1)^{|\Gamma|}} = (-1)^{n-1}(n-1)!.$$
This equality has a combinatorial interpretation: Among all connected graphs on $n$ labelled vertices, the difference between the number of those with an even number of edges and those with an odd number of edges is $(-1)^{n-1}(n-1)!$.
    
   We  show this by induction on $n$. Assume $n=1$: the only graph on 1 vertex has an even number of edges, and $(-1)^00!=1$. Assume now $n\ge 2$ and assume  the result to be true for integer partitions of $n-1$.
    
    Let $p_n$ and $d_n$ be the number of connected graphs on $n$ labelled vertices with an even number of edges and with an odd number of edges, respectively. 
    
    Let $P_n$ and $D_n$ be the number of graphs on $n$ labelled vertices with an even number of edges and with an odd number of edges, respectively. In fact
\begin{align*}
P_n &= \binom{\binom{n}{2}}{0}+\binom{\binom{n}{2}}{2}+\cdots,\\
D_n& = \binom{\binom{n}{2}}{1}+\binom{\binom{n}{2}}{3}+\cdots.
\end{align*}
    
    We can count the number of disconnected graphs with an even (odd) number of edges on $n$ labelled vertices in two ways: on one hand it is $P_n-p_n$ (respectively $D_n-d_n$), on the other hand we can count the number of rooted disconnected graphs with an even (odd) number of edges on $n$ labelled vertices (rooted means with an highlighted vertex, the root) and then divide this number by $n$, as a graph can be rooted in $n$ different ways.

    To count the number of rooted disconnected graphs, we first choose the connected component containing the root: for every possible cardinality $i=1,\ldots,n-1$ we can choose the connected component in $n\choose i$ ways, inside of which we have $i$ ways to choose the root.
    
   If we want the graph to have an even number of edges then either both the connected component of the root and the rest have an even number of edges, or they have both an odd number of edges.
    Similarly,   if we want the graph to have an odd number of edges then either the connected component of the root has an even number of edges and the rest has an odd number of edges, or viceversa. Therefore, we have
    \begin{align}\label{eq:7}
        P_n-p_n = \frac{1}{n}\sum_{i = 1}^{n-1}{i\binom{n}{i}(p_iP_{n-i} + d_iD_{n-i})},
    \end{align}
    \begin{align}\label{eq:8}
        D_n-d_n = \frac{1}{n}\sum_{i = 1}^{n-1}{i\binom{n}{i}(p_iD_{n-i} + d_iP_{n-i})}.
    \end{align}
    Taking \eqref{eq:8}-\eqref{eq:7} we get 
    \begin{align}\label{eq:9}
        p_n-d_n+D_n-P_n = \frac{1}{n}\sum_{i = 1}^{n-1}{i\binom{n}{i}(D_{n-i}-P_{n-i})(p_i-d_i)}.
    \end{align}
    
    From the Binomial theorem, we have
    $$\sum_{k=0}^{m}{(-1)^k\binom{m}{k}}=0\quad \forall m>0.$$
 This implies $D_n-P_n=0$ whenever $\binom{n}{2}>0$, that is, for every $n\geq 2$. For $n=1$, $D_1-P_1=-1$, because the only graph on 1 vertex has an even number of edges.
    So~\eqref{eq:9} becomes
    $$p_n-d_n = \frac{1}{n}(n-1)n(-1)(p_{n-1}-d_{n-1})=(-1)^{n-1}(n-1)!,$$
    where the last equality follows  by our inductive hypothesis.
   
    Now that we have concluded the special case, we need to deduce the general case. Given $\pi = (1^{c_1},\dots,n^{c_n})$, the number of ways that the set $\{1,\dots,n\}$ can be partitioned into $\pi$ is 
    $$\frac{n!}{\prod_{i = 1}^{n}{c_i!(i!)^{c_i}}}.$$
    Now, given a specific partition $\mathcal{P}$ of $\{1,\dots,n\}$ realizing $\pi$, the sum among all graphs $\Gamma$ with connected components $\mathcal{P}$ is
    \begin{align*}
 \sum_{\substack{\Gamma\subseteq\binom{\{1,\dots,n\}}{2}\\\hbox{with conn. comp.s}\\ \mathcal{P}}}{(-1)^{|\Gamma|}}&=\prod_{\mathcal{S}\in \mathcal{P}}{\sum_{\substack{\gamma\subseteq\binom{\mathcal{S}}{2}\\ \pi(\gamma)=(|\mathcal{S}|) }}{(-1)^{|\gamma|}}} = \prod_{\mathcal{S}\in \mathcal{P}}{(-1)^{|\mathcal{S}|-1}(|\mathcal{S}|-1)!}\\
       & = \prod_{i = 1}^{n}{(-1)^{(i-1)c_i}(i-1)!^{c_i}} = (-1)^{n-\sum_{i = 1}^{n}{c_i}}\prod_{i = 1}^{n}{(i-1)!^{c_i}}.   
    \end{align*}
    
    In conclusion
    \begin{align*}
        \sum_{\substack{\Gamma\subseteq\binom{\{1,\dots,n\}}{2}\\\pi(\Gamma)=\pi}}{(-1)^{|\Gamma|}} &= \frac{n!}{\prod_{i = 1}^{n}{c_i!i!^{c_i}}} (-1)^{n-\sum_{i = 1}^{n}{c_i}}\prod_{i=1}^{n}{(i-1)!^{c_i}}=\\
        &= (-1)^{n-\sum_{i=1}^{n}{c_i}}\frac{n!}{\prod_{i=1}^{n}{i^{c_i}c_i!}}.\qedhere
    \end{align*}
\end{proof}

Using Lemma~\ref{lemma:3} in~\eqref{eq:33vera}, we  get
\begin{align}\label{eq:10}
h_\ell = \sum_{\substack{\pi\hbox{ partition of }n\\
\pi=(1^{c_1},2^{c_2},\ldots,n^{c_n})}}(-1)^{n-\sum_{i=1}^nc_i}\frac{n!}{\prod_{i=1}^ni^{c_i}c_i!}f_\pi^\ell.
\end{align}

\begin{lemma}\label{lemma:yoyo}
Given an integer partition $\pi=(1^{c_1},\ldots,n^{c_n})$, we have
\begin{align}\label{eq:6}
f_\pi^\ell&=\left(\sum_{\substack{\eta\hbox{ partition of }k\\
\eta=(1^{b_1},2^{b_2},\ldots,k^{b_k})}}\prod_{j=1}^k{c_j\choose b_j}\right)^\ell.
\end{align}
\end{lemma}
\begin{proof}
 Given $\Gamma\subseteq{\{1,\ldots,n\} \choose 2}$ with $\pi(\Gamma)=\pi$ and connected components $X_1,\ldots,X_t$, we want to calculate how many elements of $F^\ell$ are fixed by $\langle (i\,j)\mid\{i,j\}\in\Gamma\rangle$. This is the number of $k$-subsets of $\{1,\ldots,n\}$ that are fixed by $\langle (i\,j)\mid\{i,j\}\in\Gamma\rangle$, raised to the $\ell$, since in order for an element of $F^\ell$ to be fixed, each of its coordinates has to be fixed.

    To calculate this number, we first notice that in order for a $k$-subset $S$ to be fixed, either $X_i\subseteq S$ or $X_i\cap S = \emptyset$. 
    Thus the number we are looking for is the number of ways we can create a set of order $k$ combining different $X_i$s.
\end{proof}

From~\eqref{eq:10} and~\eqref{eq:6}, we find the beautiful equality
\begin{align*}
h_\ell&=\sum_{\substack{\pi\hbox{ partition of }n\\
\pi=(1^{c_1},2^{c_2},\ldots,n^{c_n})}}(-1)^{n-\sum_{i=1}^nc_i}\frac{n!}{\prod_{i=1}^ni^{c_i}c_i!}\left(\sum_{\substack{\eta\hbox{ partition of }k\\
\eta=(1^{b_1},2^{b_2},\ldots,k^{b_k})}}\prod_{j=1}^k{c_j\choose b_j}\right)^\ell.
\end{align*}
Now, Theorem~\ref{thrm:main} follows from~\eqref{key}.
\end{document}